\documentclass{amsart}

\newtheorem{theorem}[equation]{Theorem}
\newtheorem{lemma}[equation]{Lemma}
\newtheorem{corollary}[equation]{Corollary}
\newtheorem{proposition}[equation]{Proposition}

\numberwithin{equation}{section}

\usepackage{amscd,amsmath,amssymb,verbatim}

\begin{document}

\title[$A$-hypergeometric series and the Hasse-Witt matrix]{$A$-hypergeometric series and the Hasse-Witt matrix of a hypersurface}
\author{Alan Adolphson}
\address{Department of Mathematics\\
Oklahoma State University\\
Stillwater, Oklahoma 74078}
\email{adolphs@math.okstate.edu}
\author{Steven Sperber}
\address{School of Mathematics\\
University of Minnesota\\
Minneapolis, Minnesota 55455}
\email{sperber@math.umn.edu}
\date{\today}
\keywords{}
\subjclass{}
\begin{abstract}
We give a short combinatorial proof of the generic invertibility of the Hasse-Witt matrix of a projective hypersurface.  We also examine the relationship between the Hasse-Witt matrix and certain $A$-hypergeometric series, which is what motivated the proof.  
\end{abstract}
\maketitle

\section{Introduction}

The Hasse-Witt matrix of a projective hypersurface over a finite field of characteristic $p$ gives essentially complete mod $p$ information about the zeta function of the hypersurface.  In particular, it determines the number of unit roots of the zeta function.  L. Miller\cite{M} was the first to prove the generic invertibility of the Hasse-Witt matrix, which implies that for a generic hypersurface, the number of unit roots of the zeta function equals the geometric genus.   Miller's proof involved an extensive study of the action of Frobenius on the coherent cohomology of hypersurfaces.  We refer the reader to Koblitz\cite{K}, Miller\cite{M2}, and Illusie\cite{I} for further results of this type.  The purpose of this note is to give a short combinatorial proof of the generic invertibility of the Hasse-Witt matrix of a hypersurface.

We give the proof of generic invertibility in Section 2.  In Section 3, we discuss the relationship between the Hasse-Witt matrix and certain $A$-hypergeometric series that satisfy the Picard-Fuchs equation of the hypersurface.  This relationship motivated the proof in Section 2 and will be the basis for future work on the unit roots of the zeta function.  

\section{Generic invertibility}

Consider a homogeneous polynomial of degree $d$ in $n+1$ variables with indeterminate coefficients:
\begin{equation}
f_\Lambda(x_0,\dots,x_n) =  \sum_{k=1}^N \Lambda_k x^{{\bf a}_k}\in {\mathbb Z}[\Lambda_1,\dots,\Lambda_N][x_0,\dots,x_n].
\end{equation}
We write ${\bf a}_k = (a_{0k},\dots,a_{nk})$ with $\sum_{i=0}^n a_{ik} = d$.  Let ${\mathbb N}$ denote the set of nonnegative integers and put
\[ U = \bigg\{ u=(u_0,\dots,u_n)\in{\mathbb N}^{n+1}\mid \text{$\sum_{i=0}^n u_i=d$ and $u_i>0$ for $i=0,\dots,n$}\bigg\}, \]
i.~e., $\{x^u\mid u\in U\}$ is the set of all monomials of degree $d$ that are divisible by the product $x_0\cdots x_n$.  We assume throughout that $d\geq n+1$, which implies $U\neq\emptyset$.

Let $p$ be a prime number.  We define a matrix of polynomials with rows and columns indexed by $U$: Let $A(\Lambda) = [A_{uv}(\Lambda)]_{u,v\in U}$, where
\begin{equation}
 A_{uv}(\Lambda) = \sum_{\substack{e_1+\cdots+e_N=p-1\\ \sum_{k=1}^N e_k{\bf a}_k = pu-v}} \frac{(p-1)!}{e_1!\cdots e_N!} \Lambda_1^{e_1}\cdots \Lambda_N^{e_N}\in {\mathbb Z}[\Lambda_1,\dots,\Lambda_N], 
\end{equation}
i.~e., $A_{uv}(\Lambda)$ is the coefficient of $x^{pu-v}$ in $f_\Lambda(x)^{p-1}$.  Let $\bar{A}_{uv}(\Lambda)\in{\mathbb F}_p[\Lambda]$ be the reduction mod $p$ of $A_{uv}(\Lambda)$ and put $\bar{A}(\Lambda) = 
[\bar{A}_{uv}(\Lambda)]_{u,v\in U}$.  

Let ${\mathbb F}_q$ be the finite field of $q$ elements, $q=p^a$, and let $\lambda = (\lambda_1,\dots,\lambda_N)\in{\mathbb F}_q^N$.  By N. Katz\cite[Algorithm 2.3.7.14]{Ka}, the matrix $\bar{A}(\lambda)$ is the Hasse-Witt matrix (relative to a certain basis) of the hypersurface in ${\mathbb P}^n_{{\mathbb F}_q}$ defined by the equation $f_\lambda(x) = 0$, where
\[ f_\lambda(x_0,\dots,x_n) = \sum_{k=1}^N \lambda_kx^{{\bf a}_k}\in{\mathbb F}_q[x_0,\dots,x_n]. \]
Our main result is the following statement.
\begin{theorem}
If $U\subseteq\{{\bf a}_k\}_{k=1}^N$, then $\det\bar{A}(\Lambda)\in{\mathbb F}_p[\Lambda]$ is not the zero polynomial.
\end{theorem}

To fix ideas, we suppose that $U=\{{\bf a}_k\}_{k=1}^{M}$, $M\leq N$.  For notational convenience we set
\[ {\bf a}_k^+ = (a_{0k},\dots,a_{nk},1)\in{\mathbb N}^{n+2},\quad k=1,\dots,N. \]
Using (2.2), we can then write $A(\Lambda) = [A_{ij}(\Lambda)]_{i,j=1}^M$ with
\begin{equation}
A_{ij}(\Lambda) = \sum_{\sum_{k=1}^N e_k{\bf a}_k^+ = p{\bf a}_i^+-{\bf a}_j^+} \frac{(p-1)!}{e_1!\cdots e_N!} \Lambda_1^{e_1}\cdots \Lambda_N^{e_N}.
\end{equation}
We consider the related matrix $B(\Lambda) =  [B_{ij}(\Lambda)]_{i,j=1}^M$ defined by
\begin{equation}
B_{ij}(\Lambda) = \Lambda_i^{-p}\Lambda_j A_{ij}(\Lambda)\in {\mathbb Z}[\Lambda_1,\dots,\Lambda_{i-1},\Lambda_i^{-1},\Lambda_{i+1},\dots,\Lambda_N],
\end{equation}
i.~e., for all $i,j=1,\dots,M$ we multiply row $i$ of $A(\Lambda)$ by $\Lambda_i^{-p}$ and column $j$ of $A(\Lambda)$ by $\Lambda_j$.  This implies that $\det \bar{B}(\Lambda) = \big(\prod_{k=1}^M \Lambda_k^{-(p-1)}\big)\det \bar{A}(\Lambda)$, hence to prove Theorem 2.3 it suffices to prove that
\begin{equation}
\det \bar{B}(\Lambda)\neq 0.
\end{equation}

We prove (2.6) by studying the monomials that appear in each $B_{ij}(\Lambda)$.  For each $i=1,\dots,M$, put
\[ L_i = \bigg\{ l=(l_1,\dots,l_N)\in{\mathbb Z}^N\mid \text{$\sum_{k=1}^N l_k{\bf a}_k^+ = {\bf 0}$, $l_i\leq 0$, and $l_k\geq 0$ for $k\neq i$}\bigg\}. \]
An easy calculation from (2.4) and (2.5) establishes the following result.
\begin{lemma}
If the monomial $\Lambda^l$ appears in $B_{ij}(\Lambda)$, then $l\in L_i$.
\end{lemma}

\begin{lemma}
The Laurent polynomial $B_{ij}(\Lambda)$ has no constant term if $i\neq j$ and has constant term equal to $1$ if $i=j$.
\end{lemma}

\begin{proof}
If $i\neq j$, then by (2.5) all terms of $B_{ij}(\Lambda)$ contain the factor $\Lambda_j$.  If $i=j$, then taking $e_i=p-1$, $e_k=0$ for $k\neq i$, in (2.4) shows that the monomial $\Lambda_i^{p-1}$ appears in $A_{ii}(\Lambda)$.  This implies by~(2.5) that $B_{ii}(\Lambda)$ has constant term equal to~1.
\end{proof}

The following observation is the key to proving (2.6).
\begin{proposition}
Let $l^{(i)} = (l^{(i)}_1,\dots,l^{(i)}_N)\in L_i$ for $i=1,\dots,M$.  One has
\[ \sum_{i=1}^M l^{(i)} = {\bf 0} \]
if and only if $l^{(i)}={\bf 0}$ for all $i$.
\end{proposition}

Before proving Proposition 2.9, we explain how it implies (2.6).  The Laurent polynomial $\det B(\Lambda)$ is a sum of terms of the form
\begin{equation}
\pm B_{1\sigma(1)}(\Lambda)\cdots B_{M\sigma(M)}(\Lambda),
\end{equation}
where $\sigma$ is a permutation on $M$ objects.  By Lemma 2.7, each monomial in this product is of the form
\[ x^{\sum_{i=1}^M l^{(i)}},\quad l^{(i)}\in L_i\text{ for $i=1,\dots,M$.} \]
By Proposition 2.9, this equals 1 if and only if $l^{(i)}={\bf 0}$ for $i=1,\dots,M$, i.~e., the product (2.10) has a nonzero constant term if and only if each factor $B_{i\sigma(i)}(\Lambda)$ has a nonzero constant term, in which case the constant term of (2.10) is (up to sign) the product of the constant terms of the $B_{i\sigma(i)}(\Lambda)$.  By Lemma 2.8, this occurs only when $\sigma$ is the identity permutation, in which case each factor has constant term equal to~1.  We thus have the following more precise version of (2.6).
\begin{proposition}
The Laurent polynomials $\det B(\Lambda)$ and $\det\bar{B}(\Lambda)$ have constant term equal to $1$.
\end{proposition}

\begin{proof}[Proof of Proposition $2.9$]
The ``if'' part of Proposition 2.9 is clear, so consider a set $\{l^{(i)}\}_{i=1}^M$ with $l^{(i)}\in L_i$ for each $i$ satisfying
\begin{equation}
\sum_{i=1}^M l^{(i)} = {\bf 0}.
\end{equation}
Since $l^{(i)}\in L_i$, we have
\begin{equation}
\sum_{k=1}^N l^{(i)}_k{\bf a}_k^+={\bf 0}\quad\text{for  each $i\in\{1,\dots,M\}$,}
\end{equation}
which implies (since the last coordinate of each ${\bf a}_k^+$ equals 1)
\begin{equation}
\sum_{k=1}^N l^{(i)}_k = 0\quad\text{for each $i\in\{1,\dots,M\}$.}
\end{equation}
If $l^{(i)}\neq{\bf 0}$, then (by the definition of $L_i$ and (2.14)) $l^{(i)}_i<0$ and $l^{(i)}_k>0$ for at least one $k$, $k\in\{1,\dots,N\}$, $k\neq i$.  We can then solve Equation (2.13) for ${\bf a}_i^+$:
\begin{equation}
{\bf a}_i^+ = \sum_{\substack{k=1\\ k\neq i}}^N \bigg(-\frac{l^{(i)}_k}{l^{(i)}_i}\bigg){\bf a}_k^+.
\end{equation}
Equation (2.14) shows that the coefficients on the right-hand side of (2.15) are nonnegative rational numbers that sum to 1, so (2.15) implies the following statement:
\begin{lemma}
If $l^{(i)}\neq{\bf 0}$, then ${\bf a}_i^+$ is an interior point of the convex hull of the set $\{{\bf a}_k^+\mid l^{(i)}_k>0\}$.
\end{lemma}

Set $C=\{ {\bf a}_k^+\mid \text{$\l^{(i)}_k\neq0$ for some $i$, $i\in\{1,\dots,M\}$}\}$.  If ${\bf a}_k^+\not\in C$, then $l^{(i)}_k=0$ for all $i\in\{1,\dots,M\}$, so if $C=\emptyset$ then $l^{(i)}_k=0$ for all $k\in\{1,\dots,N\}$ and all $i\in\{1,\dots,M\}$, which establishes the proposition.   So assume $C\neq\emptyset$ and choose $k_0\in \{1,\dots,N\}$ such that ${\bf a}_{k_0}^+$ is a vertex of the convex hull of $C$.  If $k_0\in\{1,\dots,M\}$, then Lemma~2.16 implies that $l^{(k_0)}={\bf 0}$.  In particular, if $l^{(i)}\neq{\bf 0}$, then $i\neq k_0$.  But by the definition of $L_i$, the only coordinate of $l^{(i)}$ that can be negative is $l^{(i)}_i$, all other coordinates must be nonnegative, so $l^{(i)}\neq{\bf 0}$ implies that $l^{(i)}_{k_0}\geq 0$.  Furthermore, since ${\bf a}_{k_0}^+\in C$, we must have $l^{(i)}_{k_0}>0$ for some~$i$.  
But this implies that $\sum_{i=1}^M l^{(i)}_{k_0}>0$, contradicting~(2.12).
\end{proof}

\section{$A$-hypergeometric series}

We continue to assume that $U=\{{\bf a}_k\}_{k=1}^M$, $M\leq N$.  We recall the definition of the $A$-hypergeometric system of PDEs associated to the set $\{{\bf a}_k^+\}_{k=1}^N$.  Let $L\subseteq{\mathbb Z}^{N}$ be the lattice of relations on this set:
\[ L= \bigg\{l=(l_1,\dots,l_N)\in{\mathbb Z}^N\mid \sum_{k=1}^N l_k{\bf a}_k^+ = {\bf 0}\bigg\}. \]
For each $l=(l_1,\dots,l_N)\in L$, we define a partial differential operator $\Box_l$ in the variables $\{\Lambda_k\}_{k=1}^N$ by
\begin{equation}
\Box_l = \prod_{l_k>0} \bigg(\frac{\partial}{\partial\Lambda_k}\bigg)^{l_k} - \prod_{l_k<0} \bigg(\frac{\partial}{\partial\Lambda_k}\bigg)^{-l_k}.
\end{equation}
We refer to these operators as box operators.
For $\beta = (\beta_0,\beta_1,\dots,\beta_{n+1})\in{\mathbb C}^{n+2}$, the corresponding Euler (or homogeneity) operators are defined by
\begin{equation}
 Z_i = \sum_{k=1}^N a_{ik}\Lambda_k\frac{\partial}{\partial\Lambda_k} - \beta_i 
\end{equation}
for $i=0,\dots,n+1$.  The {\it $A$-hypergeometric system with parameter $\beta$\/} consists of Equations (3.1) for $l\in L$ and (3.2) for $i=0,1,\dots,n+1$.  Note that a monomial $\Lambda^l$ satisfies the Euler operators (3.2) if and only if $\sum_{k=1}^N l_k{\bf a}_k^+ = \beta$.  

In \cite{AS}, we constructed certain logarithmic solutions of the box operators.  By \cite[Theorem 4.11]{AS}, the series $\log\Lambda_i + G_i(\Lambda)$ for $i=1,\dots,M$ satisfy the box operators, where
\[ G_i(\Lambda) = \sum_{\substack{l=(l_1,\dots,l_N)\in L_i\\ l_i<0}} (-1)^{-l_i-1}\frac{(-l_i-1)!}{\displaystyle \prod_{\substack{k=1\\ k\neq i}}^N l_k!}\Lambda^l\quad\text{for $i=1,\dots,M$.} \]
For $j=1,\dots,M$ we have
\begin{multline}
\frac{\partial}{\partial\Lambda_j}\bigg(\log\Lambda_i + G_i(\Lambda)\bigg) = \\
\begin{cases} {\displaystyle 
\sum_{l=(l_1,\dots,l_N)\in L_i} (-1)^{-l_i}\frac{(-l_i)!}{\displaystyle\prod_{\substack{k=1\\ k\neq i}}^N l_k!}\Lambda_1^{l_1}\cdots\Lambda_{i-1}^{l_{i-1}}\Lambda_i^{l_i-1}\Lambda_{i+1}^{l_{i+1}}\cdots\Lambda_N^{l_N}} & \text{if $j=i$,} \\
{\displaystyle \sum_{\substack{ l=(l_1,\dots,l_N)\in L_i\\ l_j>0}} (-1)^{-l_i-1}\frac{(-l_i-1)!}{\displaystyle(l_j-1)!\prod_{\substack{k=1\\ k\neq i,j}}^N l_k!}\Lambda_1^{l_1}\cdots\Lambda_{j-1}^{l_{j-1}}\Lambda_j^{l_j-1}\Lambda_{j+1}^{l_{j+1}}\cdots\Lambda_N^{l_N}} & \text{if $j\neq i$.} \end{cases}
\end{multline}
Note that the term $1/\Lambda_i$ that results from differentiating $\log\Lambda_i$ when $j=i$ occurs as the term $l={\bf 0}$ in the first summation.  Note also that the condition $l_j>0$ in the second summation implies $l_i<0$ as well since (2.14) holds for $l\in L_i$.

\begin{proposition}
For all $i,j=1,\dots,M$, the series in $(3.3)$ are solutions of the $A$-hypergeometric system with parameter $\beta=-{\bf a}_j^+$ and have integer coefficients.
\end{proposition}

\begin{proof}
Since $\partial/\partial\Lambda_j$ commutes with the box operators, it follows that these series satisfy the box operators.  Every monomial appearing in the series has exponent of the form $(l_1,\dots,l_{j-1},l_j-1,l_{j+1},\dots,l_N)$, so (2.13) implies that
\begin{equation}
(l_j-1){\bf a}_j^+  + \sum_{\substack{k=1\\ k\neq j}}^N l_k{\bf a}_k^+ = -{\bf a}_j^+,
\end{equation}
which implies that each monomial in the series satisfies the Euler operators with $\beta = -{\bf a}_j^+$.  By (2.14) we have
\[ -l_i = \sum_{\substack{k=1\\ k\neq i}}^N l_k, \]
which implies that each coefficient is (up to sign) a multinomial coefficient.
\end{proof}

We consider certain truncations of these series.  For each $r=(r_1,\dots,r_N)\in{\mathbb Z}^N$, we define a truncation operator ${\rm Trunc}_r$ on formal series by the formula
\begin{equation}
{\rm Trunc}_r\bigg( \sum_{s=(s_1,\dots,s_N)\in{\mathbb Z}^N} c_s\Lambda^s\bigg) = \sum_{\substack{pr_k\leq s_k<p(r_k+1)\\ \text{for $k=1,\dots,N$}}} c_s\Lambda^s.
\end{equation}

\begin{lemma}
If $F(\Lambda) = \sum_{s\in{\mathbb Z}^N} c_s\Lambda^s\in{\mathbb Z}[[\Lambda_1^{\pm 1},\dots,\Lambda_N^{\pm 1}]]$ satisfies the $A$-hyper\-geometric system for some parameter $\beta\in{\mathbb Z}^{n+2}$, then every truncation ${\rm Trunc}_r\big(F(\Lambda)\big)$ is a mod $p$ solution of that system.
\end{lemma}

\begin{proof}
The Euler operators map each monomial to a multiple of itself, so if $F(\Lambda)$ satisfies the Euler operators then any truncation of $F(\Lambda)$ also satisfies the Euler operators.  An easy calculation shows that differentiation with respect to any variable and ${\rm Trunc}_r$ for any $r\in{\mathbb Z}^N$ commute modulo $p$:
\[ \frac{\partial}{\partial\Lambda_i}\bigg({\rm Trunc}_r\big(F(\Lambda)\big)\bigg)\equiv {\rm Trunc}_r\bigg(\frac{\partial}{\partial\Lambda_i}\big(F(\Lambda)\big)\bigg) \pmod{p}. \]
It follows that if $F(\Lambda)$ satisfies the box operators, then ${\rm Trunc}_r\big(F(\Lambda)\big)$ satisfies the box operators modulo $p$.
\end{proof}

To make the connection with the Hasse-Witt matrix, we consider the truncation operator ${\rm Trunc}_{\rho^{(i)}}$ for $i=1,\dots,M$, where $\rho^{(i)} = (0,\dots,0,-1,0,\dots,0)$ and the ``-1'' occurs in the $i$-th coordinate.
\begin{proposition}
For $i,j=1,\dots,M$ we have
\[ A_{ij}(\Lambda)\equiv -\Lambda_i^p\,{\rm Trunc}_{\rho^{(i)}}\bigg(\frac{\partial}{\partial\Lambda_j}\bigg(\log\Lambda_i + G_i(\Lambda)\bigg)\bigg) \pmod{p}. \]
\end{proposition}

\begin{proof}
It follows from (3.3) that
\begin{equation}
{\rm Trunc}_{\rho^{(i)}}\bigg(\frac{\partial}{\partial\Lambda_j}\bigg(\log\Lambda_i + G_i(\Lambda)\bigg)\bigg) = 
{\sum}'(-1)^{-e_i-1}{\displaystyle \frac{(-e_i-1)!}{\prod_{\substack{k=1\\k\neq i}}^N e_k!}}\Lambda_1^{e_1}\cdots\Lambda_N^{e_N}
\end{equation}
where the symbol ${\sum}'$ denotes a sum over $(e_1,\dots,e_N)$ satisfying $\sum_{k=1}^N e_k{\bf a}_k^+ = -{\bf a}_j^+$, $-p\leq e_i\leq -1$, and $0\leq e_k\leq p-1$ for $k\neq i$.  Since 
\[ (-1)^{-e_i-1}(-e_i-1)!\equiv (p+e_i)!^{-1}\pmod{p}, \]
Equation (3.9) implies
\begin{equation}
{\rm Trunc}_{\rho^{(i)}}\bigg(\frac{\partial}{\partial\Lambda_j}\bigg(\log\Lambda_i + G_i(\Lambda)\bigg)\bigg) \equiv
{\sum}'{\displaystyle \frac{\Lambda_1^{e_1}\cdots\Lambda_N^{e_N}}{(p+e_i)!\prod_{\substack{k=1\\k\neq i}}^N e_k!}} \pmod{p}.
\end{equation}
We have $p+\sum_{k=1}^N e_k=p-1$ and $(p-1)!\equiv -1\pmod{p}$, so multiplying (3.10) by $(p-1)!\Lambda_i^p$ and comparing with (2.4) establishes the proposition.
\end{proof}

Combining Proposition 3.4, Lemma 3.7, and Proposition 3.8 gives the following result.
\begin{corollary}
For all $i,j=1,\dots,M$, the coefficient $\bar{A}_{ij}(\Lambda)$ of the Hasse-Witt matrix is a mod $p$ solution of the $A$-hypergeometric system with parameter $\beta = -{\bf a}_j^+$.
\end{corollary}

{\bf Remark 1:}  One can also deduce Corollary 3.11 as a consequence of Equation~(2.4) and \cite[Theorem~2.7 or Theorem~3.1]{AS1}.  The proof given here shows a more direct connection with integral $A$-hypergeometric series.

{\bf Remark 2:}  Let $U\subseteq{\mathbb P}^n_{{\mathbb C}(\Lambda)}$ be the complement of the hypersurface $f_\Lambda = 0$.  The differential operators $\partial/\partial\Lambda_k$ act on the cohomology classes in $H^n_{\rm DR}(U/{\mathbb C}(\Lambda))$ via the Gauss-Manin connection.  Put
\[ \Theta = \sum_{i=0}^n (-1)^i \frac{dx_0}{x_0}\wedge\cdots\widehat{\frac{dx_i}{x_i}}\cdots\wedge\frac{dx_n}{x_n}. \]
The cohomology classes with simple poles along $f_\Lambda = 0$ are spanned by the $n$-forms given in homogeneous coordinates as
\begin{equation}
 \frac{x^{{\bf a}_j}\Theta}{f_\Lambda(x_0,\dots,x_n)} \quad \text{for $j=1,\dots,M$.} 
\end{equation}
It can be shown that the cohomology class determined by the $n$-form (3.12) satisfies the $A$-hypergeometric system with parameter $\beta = -{\bf a}_j^+$.

\end{document}